\documentclass[11pt]{amsart}

\usepackage{hyperref}
\hypersetup{bookmarksdepth=3}
  
\usepackage{geometry}
\usepackage{amsmath,amssymb}
\usepackage{
mathrsfs}
\usepackage{enumerate}
\usepackage{esint}

\usepackage{tensor}

\usepackage{amsmath}

\usepackage{IEEEtrantools}

\usepackage{dirtytalk}

\usepackage{mathtools}

\usepackage{amssymb,amsmath,amsthm}

\newtheorem{theorem}{Theorem}

\newtheorem{definition}[theorem]{Definition}

\theoremstyle{remark}\newtheorem{remark}[theorem]{Remark}

\usepackage{graphicx}

\newcommand{\R}{{\mathbb R}}

\newcommand{\Z}{{\mathbb Z}}

\begin{document}

{\let\thefootnote\relax\footnote{Date: \today. 

\textcopyright 2019 by the authors. Faithful reproduction of this article, in its entirety, by any means is permitted for noncommercial purposes.}}

\title{On the global wellposedness of the Klein-Gordon equation for initial data in modulation spaces.}


\author{L. Chaichenets}
\address{leonid chaichenets, department of mathematics, institute for analysis, karlsruhe institute of technology, 76128 karlsruhe, germany }
\email{leonid.chaichenets@kit.edu}

\author{N. Pattakos}
\address{nikolaos pattakos, department of mathematics, institute for analysis, karlsruhe institute of technology, 76128 karlsruhe, germany }
\email{nikolaos.pattakos@kit.edu}

\begin{abstract}
{We prove global wellposedness of the Klein-Gordon equation with power non\-linearity $|u|^{\alpha-1}u$, where $\alpha\in\left[1,\frac{d}{d-2}\right]$, in dimension $d\geq3$ with initial data in $M_{p, p'}^{1}(\R^d)\times M_{p,p'}(\R^d)$ for $p$ sufficiently close to $2$. The proof is an application of the high-low method described by Bourgain in \cite{BOU} where the Klein-Gordon equation is stu\-died in one dimension with cubic nonlinearity for initial data in Sobolev spaces.}
\end{abstract}

\maketitle
\pagestyle{myheadings}

\begin{section}{Introduction and Main Results}
\markboth{\normalsize L. Chaichenets and N. Pattakos }{\normalsize  Klein-Gordon equation in modulation spaces}

We are interested in the following Cauchy problem 
\begin{equation} 
\label{maineq}
\left\{
\begin{IEEEeqnarraybox}[][c]{rCl}
u_{tt}(t,x)-\Delta u(t,x)+u+|u|^{\alpha-1}u & = & 0, \\
u(0,x) & = & \phi(x),\\
u_{t}(0,x) & = & \psi(x),
\end{IEEEeqnarraybox}
\right.
\end{equation}
with initial data in modulation spaces $M_{p,q}^{s}(\R^d)\times M^{s-1}_{p,q}(\R^d)$ (see Definition \ref{defmod}). Before we state our main results let us mention that equation \eqref{maineq} plays a fundamental role in the study of the kinematics of relativistic particles (see \cite{SEG}) and has been studied extensively, see e.g. \cite{BW}, \cite{GR1}, \cite{GR2}, \cite{GR3}, \cite{KAP1}, \cite{KAP2}, \cite{MS1}, \cite{MS2}, \cite{HP} and \cite{ZZ}. This selection is far from exhaustive and we refer the interested reader to the references cited in the previously mentioned papers. The Cauchy problem \eqref{maineq} is globally wellposed in $H^{1}(\R^d)\times L^{2}(\R^d)$ for energy subcritical nonlinearities, i.e.
\begin{equation}
\label{dd}
1\leq\alpha<1+\frac{4}{d-2},
\end{equation}
and we have the following theorem from \cite[Proposition 3.2]{GV1}
\begin{theorem}
\label{GVth}
Let $t_{0}\in \R$ and $(\phi,\psi)\in H^{1}(\R^d)\times L^{2}(\R^d)$. Then, equation \eqref{maineq} has a unique global solution $u$ in $L^{\infty}_{loc}(\R, H^{1}(\R^d))$ with $u(t_{0})=\phi$, $\partial_{t}u(t_{0})=\psi$. In addition, $u\in C(\R, H^{1}(\R^d))\cap C^{1}(\R,L^{2}(\R^d))$ and satisfies the energy conservation
\begin{equation}
\label{energycon}
E(u(t), \partial_{t}u(t))=E(\phi,\psi),
\end{equation}
where the energy is defined as
\begin{equation}
\label{energydef}
E(\phi,\psi)=\|\psi\|_{L^{2}(\R^d)}^{2}+\|\phi\|_{H^{1}(\R^d)}^{2}+\frac1{\alpha+1}\ \|\phi\|_{L^{\alpha+1}(\R^d)}^{\alpha+1}.
\end{equation}
\end{theorem}
Let us mention that the solution to the linear Klein-Gordon equation
\begin{equation} 
\label{maineqlin}
\left\{
\begin{IEEEeqnarraybox}[][c]{rCl}
u_{tt}(t,x)-\Delta u(t,x)+u & = & F(t,x), \\
u(0,x) & = & \phi(x),\\
u_{t}(0,x) & = & \psi(x),
\end{IEEEeqnarraybox}
\right.
\end{equation}
is formally given by the formula
\begin{equation}
\label{duhlin}
u(t,x)=K'(t)\phi(x)+K(t)\psi(x)+\int_{0}^{t}K(t-\tau)F(\tau,x)\ d\tau,
\end{equation}
where 
\begin{equation}
\label{kern2}
K(t)=\frac{\sin t(I-\Delta)^{\frac12}}{(I-\Delta)^{\frac12}}\ \mbox{and}\ K'(t)=\cos t(I-\Delta)^{\frac12}.
\end{equation}
Notice that the (semi-)group given by
\begin{equation}
\label{semig}
\mathbb K(t)=
\begin{pmatrix} 
K'(t) & K(t) \\
(\Delta-Id)K(t) & K'(t)
\end{pmatrix}\qquad \forall t\in\R
\end{equation}
is isometric on $H^{1}(\R^d)\times L^{2}(\R^d)$ with the domain of its generator being 
$$\mathcal D=H^{2}(\R^d)\times H^{1}(\R^d).$$

Next, we need the following definition of admissible pairs for the Klein-Gordon equation from \cite[Definition 6.3.1]{OV}:
\begin{definition}
\label{def1}
We say that the pair $(q,r)\in[2,\infty]\times[2,\infty)$ is (wave-)admissible if 
\begin{equation}
\label{adm}
\frac1{q}+\frac{d-1}{2r}\leq\frac{d-1}4.
\end{equation}
We will denote by $q_{a}(r)=q$ the unique solution of the \emph{gap condition}
\begin{equation}
\label{gapcon}
\frac1{q}+\frac{d}{r}=\frac{d}2-1.
\end{equation}
\end{definition}

Our goal is to study \eqref{maineq} for initial data in modulation spaces, whose definition is given below.
\begin{definition}
\label{defmod}
Consider $s\in\R$ and $p,q\in[1,\infty]$. We say that a distribution $f\in S'(\R^d)$ belongs to the Banach space $M_{p,q}^{s}(\R^d)$ if 
\begin{equation}
\label{modul1}
\|f\|_{M_{p,q}^{s}}=\Big\|\Big\{\langle k\rangle^{s}\|\Box_{k}f\|_{p}\Big\}_{k\in\Z^d}\Big\|_{l^{q}(\Z^d)}<\infty,
\end{equation}
where $\langle k\rangle=(1+|k|^{2})^{\frac12}$ is the Japanese bracket and the operators $\Box_{k}$ are the isometric decomposition operators given by
\begin{equation}
\label{boxx}
\Box_{k}=\mathcal F^{-1}\sigma_{k}\mathcal F,
\end{equation}
and $\{\sigma_{k}(\cdot)=\sigma_{0}(\cdot-k)\}_{k\in\Z^d}$ is a smooth partition of unity with $\sigma_{0}$ supported in the ball $B(0,\sqrt{d})$ and equal to $1$ in the unit cube $Q=[-\frac12,\frac12)^{d}$.  
\end{definition}
These spaces first appeared in \cite{FEI} and since then they have become canonical for both time-frequency and phase-space analysis. It can be proved that different choices of the function $\sigma_{0}$ lead to equivalent norms in $M^{s}_{p,q}(\mathbb R^{d})$ (see e.g. \cite[Proposition 2.9]{CHA} and \cite[Proposition 3.4]{BLB}). When $s=0$ we denote the space $M^{0}_{p,q}(\mathbb R^{d})$ by $M_{p,q}(\mathbb R^{d})$. In the special case where $p=q=2$ we have $M_{2,2}^{s}(\R^{d})=H^{s}(\R^{d})$ where
\begin{equation}
\label{sobsob}
H^{s}(\R^d)=\Big\{f\in S'(\R^d)\ \Big|\ \|f\|_{H^{s}(\R^d)}=\Big(\int_{\R^d}\langle\xi\rangle^{2s}|\hat{f}(\xi)|^{2}d\xi\Big)^{\frac12}<\infty\Big\}.
\end{equation}

Every time we write $\|f\|_{p}$ or $\|f\|_{L^{p}}$ we mean the usual $p$-norms in the Lebesgue spaces $L^{p}(\R^d)$. In addition, for a given interval $I\subset\R$ we use the notation $\|f\|_{L^{p}_{I}}$ for the $L^{p}$ norm of $f$ over $I$. Finally, for $p\in[1,\infty]$ we denote by $p'$ the conjugate exponent of $p$ that is the number that satisfies $\frac1{p}+\frac1{p'}=1$. 

Our main results are the following theorems.
\begin{theorem}
\label{maintheoremloc}
Assume that the dimension $d\geq3$ and the power nonlinearity in \eqref{maineq} satisfies\footnotemark $\alpha\in \left[1,\frac{d}{d-3}\right]\cap\R$.
Then, the Cauchy problem \eqref{maineq} with initial data 
$$(\phi,\psi)\in H^{1}(\R^d)\times L^{2}(\R^d)+M_{2\alpha,(2\alpha)'}^{1}(\R^d)\times M_{2\alpha,(2\alpha)'}(\R^d)$$
 is locally wellposed and the solution $u$ lies in 
\begin{equation}
\label{propu}
C([0,T], H^{1}(\R^d))\cap L^{q_{a}(2\alpha)}([0,T], L^{2\alpha}(\R^d))+C([0,T], M^{1}_{2\alpha, (2\alpha)'}(\R^d))
\end{equation}
with $\partial_{t}u\in C([0,T],L^{2}(\R^d))+C([0,T],M_{2\alpha,(2\alpha)'}(\R^d))$, where the guaranteed time of existence $T=T(\|(\phi,\psi)\|_{H^{1}(\R^d)\times L^{2}(\R^d)+M_{2\alpha,(2\alpha)'}^{1}(\R^d)\times M_{2\alpha,(2\alpha)'}(\R^d)})>0$. 

Moreover, if $T_{*}$ is the maximal time of existence then the blowup alternative holds, i.e.
\begin{equation}
\label{blowalt}
T_{*}<\infty\ \Rightarrow\ \limsup_{t\to T_{*}^{-}}\|(u(t,\cdot),\partial_{t}u(t,\cdot))\|_{H^{1}(\R^d)\times L^{2}(\R^d)+M_{2\alpha,(2\alpha)'}^{1}(\R^d)\times M_{2\alpha,(2\alpha)'}(\R^d)}=\infty.
\end{equation}
\end{theorem}

\footnotetext{In the case $d=3$ the expression $\frac{d}{d-3}$ shall be understood as $\infty$ and no upper bound on $\alpha$ is required.}

\begin{remark}
\label{remlast5}
The restrictions on $\alpha$ appear because of Strichartz estimates. For more details, see Remark \ref{alpharestric} after the proof of Theorem \ref{maintheoremloc}.
\end{remark}

\begin{theorem}
\label{maintheoremglo}
For $\alpha\in\left[1,\frac{d}{d-2}\right]$ and \emph{real} initial data $(\phi,\psi)\in M_{p, p'}^{1}(\R^d)\times M_{p,p'}(\R^d)$ with $p\in(2,p_{\text{max}})$, where we have
\begin{equation}
\label{kkpp}
p_{\text{max}}=\begin{cases} 
2\alpha\ \frac{(\alpha+1)(\alpha-2)+2}{\alpha(\alpha+1)(\alpha-2)+2}, & \text{if } \alpha \in \left(2,\infty\right) \\
2 \alpha, & \text{if } \alpha \in \left[1, 2\right]
\end{cases}
\end{equation}
the local solution $u$ of \eqref{maineq} constructed in Theorem \ref{maintheoremloc} extends globally and lies in
\begin{equation}
\label{propu2}
u\in C(\R, H^{1}(\R^d))+C(\R, M^{1}_{2\alpha, (2\alpha)'}(\R^d)),
\end{equation}
with $\partial_{t}u\in C(\R, L^{2}(\R^d))+C(\R, M_{2\alpha, (2\alpha)'}(\R^d))$. 
\end{theorem}

\begin{remark}
To the best of the authors' knowledge Theorem \ref{maintheoremglo} is the first global wellposedness result for the Klein-Gordon equation on modulation spaces (which do not coincide with Sobolev spaces).
\end{remark}

\begin{remark}
\label{restricp}
The restriction $\alpha\in\left[1,\frac{d}{d-2}\right]$ appears because in the proof of Theorem \ref{maintheoremglo} we estimate an $L^{2\alpha}(\R^d)$ norm of a function by its $H^{1}(\R^d)$ norm. Thus, $2\alpha\in[2,\frac{2d}{d-2}]$ by the assumptions of the Sobolev embedding \eqref{embeddsob}. 

The restrictions on $p$ arise from \eqref{alphareq} where we make use of the identities $\tilde{\alpha}=\frac{\theta}{\theta-1}$ and $\frac1{p}=\frac{1-\theta}{2}+\frac{\theta}{2\alpha}$ in order to obtain the upper bound for $p$ in terms of the nonlinearity $\alpha$.
\end{remark}

\begin{remark}
\label{dimensionsmall}
In dimensions $d \in \left\{1, 2\right\}$ Theorem \ref{maintheoremglo} holds with no restrictions on $\alpha$ and the restriction $p\in(2,p_{\text{max}})$ with $p_{\text{max}}$  given by \eqref{kkpp} remains. This is due to the fact that one has enough control from the Sobolev embeddings and does not require Strichartz estimates for the local wellposedness argument. The proof of the global existence remains the same and is presented in the next section.
\end{remark}

\begin{remark}
\label{HL47}
The method used to prove Theorem \ref{maintheoremglo}, i.e. the high-low method, was used in \cite{CHA} and \cite{CHKP} to study the NLS equation
\begin{equation}
\label{mainNLS}
\begin{cases} iu_{t}+\Delta u\pm|u|^{\alpha-1}u=0 &,\ (t,x)\in\mathbb R^{d+1}\\
u(0,x)=u_{0}(x) &,\ x\in\mathbb R^d\\
\end{cases}
\end{equation}
with initial data in modulation spaces $M_{p,p'}(\R^d)$ with $p$ close to $2$. Global existence was obtained for the cubic nonlinearity $\alpha=3$ and $d=1$ in \cite{CHKP} and then this result was generalised for all $\alpha\in(1,1+\frac4{d})$ and $d\geq1$ in \cite{CHA}.
\end{remark}
At the heart of the high low method is the following idea. The initial datum and the solution are split into two parts, namely the good part (low frequencies) and the bad part (high frequencies), such that the linear propagation of the bad part shall not pose a problem, whereas the nonlinear interaction of the high and low frequencies can be controlled by the nonlinear smoothing effect inherent to the PDE.

Before we proceed to the proofs of Theorems \ref{maintheoremloc} and \ref{maintheoremglo} let us state some known facts which are going to be used in the proofs of the main theorems in the next and last section. 

By \cite[Proposition 2.7]{BH} it is known that for any $1<p\leq\infty$ we have the embedding $M_{p,1}(\R^d)\hookrightarrow L^{p}(\R^d)\cap L^{\infty}(\R^d)$ which together with the fact that $M_{2,2}(\R^d)=L^{2}(\R^d)$ and interpolation, imply that for any $p\in[2,\infty]$ we have the embedding 
\begin{equation}
\label{embedd}
M_{p,p'}(\R^d)\hookrightarrow L^{p}(\R^d).
\end{equation}

In \cite[Theorem 6.6]{FEI} (see also \cite[Proposition 2.4]{BH}) it is proved that for all $\sigma, s\in\R$ and $p,q\in[1,\infty]$ the map $(Id-\Delta)^{\frac{\sigma}2}$ is an isomorphism from $M_{p,q}^{s}(\R^d)$ onto $M_{p,q}^{s-\sigma}(\R^d)$. We will use this fact in Section \ref{mainsec} for $\sigma=1$. 

For $d\geq3$ by the Sobolev embedding theorem in \cite[Theorem 9.9]{BRE} we have that
\begin{equation}
\label{embeddsob}
H^{1}(\R^d)\hookrightarrow L^{p}(\R^d),
\end{equation}
for all $p\in[2,\frac{2d}{d-2}]$.

If $u$ is the solution to \eqref{duhlin} with initial data $(\phi,\psi)\in H^{1}(\R^d)\times L^{2}(\R^d)$ then we have the following well known Strichartz estimate (see \cite[Corollary 1.3]{KT}, \cite[Theorem 6.4.1]{OV}, \cite[Estimate 1.1]{ZZ} and \cite[Section 2.2]{GGr})
\begin{eqnarray}
\label{Strr}
\|u\|_{L^{q}([0,T],L^{r}(\R^d))}+\|u\|_{C([0,T], H^{1}(\R^d))}+\|\partial_{t}u\|_{C([0,T],L^{2}(\R^d))} \\ \lesssim \|\phi\|_{H^{1}(\R^d)}+\|\psi\|_{L^{2}(\R^d)}+\|F\|_{L^{1}([0,T],L^{2}(\R^d))},\nonumber
\end{eqnarray}
for every admissible pair $(q,r)$ that satisfies the gap condition described in Definition \ref{def1}, i.e. 
$$q=\left(d\left(\frac12-\frac1{r}\right)-1\right)^{-1},\qquad r\in\left[\left(\frac12-\frac1{d-1}\right)^{-1},\frac{2(d+1)}{d-3}\right].$$

The operators $K(t)$ and $K'(t)$ defined in \eqref{kern2} are bounded on modulation spaces (see \cite[Corollary 3.2]{CN}) . More precisely, for $s\in\R$, $p,q\in[1,\infty]$ and every $T>0$ there is a positive constant $C_{T}=C_{T}(p,q,s)$ such that 
\begin{equation}
\label{boundd}
\|K(t)f\|_{M_{p,q}^{s+1}(\R^d)}+\|K'(t)f\|_{M_{p,q}^{s}(\R^d)}\leq C_{T}\|f\|_{M_{p,q}^{s}(\R^d)}
\end{equation}
for all $f\in M_{p,q}^{s}(\R^d)$ and all $|t|\leq T$. 

\end{section}

\begin{section}{Proofs of the Main Theorems}
\label{mainsec}
\begin{proof}[Proof of Theorem \ref{maintheoremloc}]
Since our initial data $(\phi,\psi)$ lies in the space
$$X\coloneqq H^{1}(\R^d)\times L^{2}(\R^d)+M_{2\alpha,(2\alpha)'}^{1}(\R^d)\times M_{2\alpha,(2\alpha)'}(\R^d)$$
we may write $(\phi,\psi)=(\phi_{0},\psi_{0})+(\tilde{\phi}_{0},\tilde{\psi}_{0})$ with
$$(\phi_{0},\psi_{0})\in H^{1}(\R^d)\times L^{2}(\R^d)\ \mbox{and}\ (\tilde{\phi}_{0},\tilde{\psi}_{0})\in M_{2\alpha,(2\alpha)'}^{1}(\R^d)\times M_{2\alpha,(2\alpha)'}(\R^d).$$

In the following we present the Banach contraction argument for the first coordinate of the solution $(u,\partial_{t}u)$, i.e. for $u$ only, since by \eqref{Strr} the argument for $\partial_{t}u$ is similar. In addition, we only treat the case where $\alpha\in\left(\frac{d}{d-2},\frac{d}{d-3}\right]$, since in the remaining interval $\left[1,\frac{d}{d-2}\right]$ Strichartz estimates are not required to finish the argument.

We are going to work in the Banach space $X(T)=X_{1}(T)+X_{2}(T)$ with
$$X_{1}(T)=L^{\infty}([0,T], H^{1}(\R^d))\cap L^{q_{a}(2\alpha)}([0,T], L^{2\alpha}(\R^d))$$
and
$$X_{2}(T)=L^{\infty}([0,T], M_{2\alpha, (2\alpha)'}(\R^d)).$$
The norm in $X(T)$ is given by
\begin{equation}
\label{splitX}
\|u\|_{X(T)}=\inf_{\substack{u=u_{1}+u_{2}\\ u_{1}\in X_{1}(T),u_{2}\in X_{2}(T)}}\Big[\|u_{1}\|_{X_{1}(T)}+\|u_{2}\|_{X_{2}(T)}\Big]
\end{equation}
and the operator we are interested in is
\begin{equation}
\label{opera}
\mathcal Tu=K'(t)\phi_{0}+K(t)\psi_{0}+K'(t)\tilde{\phi}_{0}+K(t)\tilde{\psi}_{0}-\int_{0}^{t}K(t-\tau)(|u|^{\alpha-1}u)\ d\tau
\end{equation}
for $u$ in the ball $M(R,T)=\{u\in X(T) | \|u\|_{X(T)}\leq R\}$. The claim is that for some positive numbers $R$ and $T$ the operator $\mathcal T$ is a contraction in $M(R,T)$. 

We start with the self-mapping property of $\mathcal T$. Let us fix $u\in M(R,T)$ and consider a splitting $u=v+w$ with $v\in X_{1}(T)$ and $w\in X_{2}(T)$. 

For the linear evolution part of $\mathcal T$ we have that the norm
$$\Big\|K'(t)\phi_{0}+K(t)\psi_{0}+K'(t)\tilde{\phi}_{0}+K(t)\tilde{\psi}_{0}\Big\|_{X(T)}$$
is controlled by
$$\Big\|K'(t)\phi_{0}+K(t)\psi_{0}\Big\|_{X_{1}(T)}+\Big\|K'(t)\tilde{\phi}_{0}+K(t)\tilde{\psi}_{0}\Big\|_{X_{2}(T)}.$$
For the first term we use the Strichartz estimate stated in \eqref{Strr} which implies
$$\Big\|K'(t)\phi_{0}+K(t)\psi_{0}\Big\|_{X_{1}(T)}\lesssim\|\phi_{0}\|_{H^{1}(\R^d)}+\|\psi_{0}\|_{L^{2}(\R^d)}$$
and for the second term the boundedness of $K(t)$ and $K'(t)$ stated in \eqref{boundd} implies
\begin{eqnarray*}
\Big\|K'(t)\tilde{\phi}_{0}+K(t)\tilde{\psi}_{0}\Big\|_{X_{2}(T)} &\lesssim& C(T)\Big(\|\tilde{\phi}_{0}\|_{M^{1}_{2\alpha,(2\alpha)'}}+\|\tilde{\psi}_{0}\|_{M_{2\alpha,(2\alpha)'}}\Big)\\ &\lesssim&  \|\tilde{\phi}_{0}\|_{M^{1}_{2\alpha,(2\alpha)'}}+\|\tilde{\psi}_{0}\|_{M_{2\alpha,(2\alpha)'}}
\end{eqnarray*}
where, without loss of generality, we assumed for the time of existence $T\leq1$. As the splitting of the initial data $(\phi,\psi)$ was arbitrary we have that 
\begin{equation}
\label{ffgt}
\Big\|K'(t)\phi+K(t)\psi\Big\|_{X(T)}\lesssim\|(\phi,\psi)\|_{X}.
\end{equation}
This suggests the choice 
\begin{equation}
\label{radiuu}
R\approx2\ \|(\phi,\psi)\|_{X}.
\end{equation}
Before we deal with the integral part of the operator $\mathcal T$ let us observe that since the modulation space $M_{2\alpha,(2\alpha)'}(\R^d)\hookrightarrow L^{2\alpha}(\R^d)$ (see \eqref{embedd}) and $L^{\infty}_{[0,T]}\hookrightarrow L^{q_{a}(2\alpha)}_{[0,T]}$, we trivially obtain
\begin{equation}
\label{esstimaa}
X(T)\hookrightarrow L^{q_{a}(2\alpha)}([0,T], L^{2\alpha}(\R^d)).
\end{equation} 
The integral part in \eqref{opera} is estimated in the $X_{1}(T)$ norm using \eqref{Strr} by
\begin{eqnarray*}
\Big\|\int_{0}^{t}K(t-\tau)(|u|^{\alpha-1}u)\ d\tau\Big\|_{L^{q_{a}(2\alpha)}([0,T], L^{2\alpha}(\R^d))} &\lesssim& \||u|^{\alpha-1}u\|_{L^{1}([0,T],L^{2}(\R^d))} \\ &=& \|u\|^{\alpha}_{L^{\alpha}([0,T],L^{2\alpha}(\R^d)} \\ &\leq&T^{\frac{d+2-\alpha(d-2)}{2}}\|u\|^{\alpha}_{L^{q_{a}(2\alpha)}([0,T], L^{2\alpha}(\R^d))}\\ &\lesssim& T^{\frac{d+2-\alpha(d-2)}{2}}\|u\|_{X(T)}^{\alpha} \\ &\leq& T^{\frac{d+2-\alpha(d-2)}{2}}R^{\alpha},
\end{eqnarray*}
where at the third step we used H\"older's inequality and at the forth step \eqref{esstimaa}. Thus, the operator $\mathcal T$ is a self-mapping of $M(R,T)$ if 
$$T^{\frac{d+2-\alpha(d-2)}{2}}R^{\alpha}\leq\frac{R}2$$
or equivalently,
\begin{equation}
\label{timedepnds}
T\ \lesssim \ R^{\frac{2(1-\alpha)}{d+2-\alpha(d-2)}}\ \approx\ \Big(\|(\phi,\psi)\|_{X}\Big)^{\frac{2(1-\alpha)}{d+2-\alpha(d-2)}}.
\end{equation}

For the contraction property of $\mathcal T$ we have by using the same considerations as above for $u_{1}, u_{2}\in X(T)$ and the size estimate 
\begin{equation}
\label{sizz}
||u_{1}|^{\alpha-1}u_{1}-|u_{2}|^{\alpha-1}u_{2}|\lesssim(|u_{1}|^{\alpha-1}+|u_{2}|^{\alpha-1})|u_{1}-u_{2}|
\end{equation}
that the following holds
\begin{eqnarray*}
& & \|\mathcal T(u_{1})-\mathcal T(u_{2})\|_{X_{1}(T)} \\
&\lesssim& T^{\frac{d+2-\alpha(d-2)}{2}}\Big(\|u_{1}\|_{L^{q_{a}(2\alpha)}_{[0,T]}L^{2\alpha}}^{\alpha-1}+\|u_{2}\|_{L^{q_{a}(2\alpha)}_{[0,T]}L^{2\alpha}}^{\alpha-1}\Big)\|u_{1}-u_{2}\|_{L^{q_{a}(2\alpha)}_{[0,T]}L^{2\alpha}} \\ &\lesssim& T^{\frac{d+2-\alpha(d-2)}{2}}\ R^{\alpha-1}\ \|u_{1}-u_{2}\|_{L^{q_{a}(2\alpha)}_{[0,T]}L^{2\alpha}}.
\end{eqnarray*}
Hence, by choosing a possibly smaller implicit constant we have that $\mathcal T$ is a contraction on $M(R,T)$ and the proof is complete.
\end{proof}
\begin{remark}
\label{alpharestric}
The restrictions on $\alpha$ stated in Theorem \ref{maintheoremloc} appear because of the number $q=q_{a}(2\alpha)$ which is defined by the gap condition \eqref{gapcon} and it is equal to 
\begin{equation}
\label{admq}
\frac1{q}=\frac{\alpha(d-2)-d}{2\alpha}.
\end{equation}
Since we must have $q\geq2$ we obtain the restriction $\alpha\leq\frac{d}{d-3}$.
\end{remark}

\begin{proof}[Proof of Theorem \ref{maintheoremglo}]
Consider functions $(\phi,\psi)\in M_{p, p'}^{1}(\R^d)\times M_{p,p'}(\R^d)$. Using complex interpolation (see \cite[Theorem 6.1 D]{FEI}) we write
\begin{equation}
\label{inter}
[H^{1}(\R^d), M_{r,r'}^{1}(\R^d)]_{\tilde{\theta}}=M_{p,p'}^{1}(\R^d), \ [L^{2}(\R^d), M_{r,r'}(\R^d)]_{\tilde{\theta}}=M_{p,p'}(\R^d)
\end{equation}
where $2<p<r=2\alpha<\infty$, $\tilde{\theta}\in(0,1)$ and $\frac1{p}=\frac{1-\tilde{\theta}}{2}+\frac{\tilde{\theta}}{r}$. Then, for $(\phi,\psi)\in M_{p,p'}^{1}(\R^d)\times M_{p,p'}(\R^d)$ and any large enough $N>0$ (to be found later) we decompose 
\begin{equation}
\label{splitt}
(\phi,\psi)=(\phi_{N},\psi_{N})+(\phi^{N},\psi^{N})\in H^{1}(\R^d)\times L^{2}(\R^d)+M_{r,r'}^{1}(\R^d)\times M_{r,r'}(\R^d)
\end{equation}
with
\begin{equation}
\label{bigsplit}
\|\phi_{N}\|_{H^{1}}, \|\psi_{N}\|_{L^{2}}\lesssim N^{\tilde{\alpha}},\ \|\phi^{N}\|_{M_{r,r'}^{1}}, \|\psi^{N}\|_{M_{r,r'}}\lesssim\frac1{N}
\end{equation}
where $\tilde{\alpha}=\frac{\tilde{\theta}}{1-\tilde{\theta}}$. This is possible since the complex interpolation spaces mentioned above embed in the real interpolation spaces (see \cite[Theorem 1.10.3/1]{TRTR})
$$(H^{1}(\R^d), M_{r,r'}^{1}(\R^d))_{\tilde{\theta},\infty},\ (L^{2}(\R^d), M_{r,r'}(\R^d))_{\tilde{\theta},\infty}$$
whose norm is given by the $K$ functional
$$\|u\|_{\tilde{\theta},\infty}=\sup_{t>0}\Big(\ t^{-\tilde{\theta}}\inf_{\substack{u=\phi+\psi \\ \phi\in H^{1},\ \psi\in M^{1}_{r,r'}}}\Big[\|\phi\|_{H^{1}}+t\|\psi\|_{M^{1}_{r,r'}}\Big]\Big)$$
(similarly for the pair $L^{2}$ and $M_{r,r'}$). Then for any given $N\in\R_{+}$ setting $t=N^{\tilde{\alpha}+1}$ and $\tilde{\theta}=\frac{\tilde{\alpha}}{\tilde{\alpha}+1}$ shows \eqref{bigsplit}. 

By Theorem \ref{maintheoremloc} for real initial data $(\phi,\psi)$ we know that there is a \emph{real-valued} local solution $u$ to \eqref{maineq}. From the blowup alternative we know that if the norm
\begin{equation}
\label{normblow}
\|(u(t,\cdot),\partial_{t}u(t,\cdot))\|_{H^{1}(\R^d)\times L^{2}(\R^d)+M_{2\alpha,(2\alpha)'}^{1}(\R^d)\times M_{2\alpha,(2\alpha)'}(\R^d)}
\end{equation}
does not blow up in finite time then the solution $u$ exists globally. Therefore, our goal is to show that we can control the quantity \eqref{normblow} on bounded time intervals. 

To simplify the notation we make the change of variables (as in \cite[Equation IV.2.4]{BOU})
\begin{equation}
\label{chnageof}
v=u+iB^{-1}u_{t},
\end{equation}
where $B^{2}=Id-\Delta$ and we rewrite \eqref{maineq} in the form
\begin{equation}
\label{newmaineq}
iv_{t}-Bv-B^{-1}(|\operatorname{Re} v|^{\alpha-1}\operatorname{Re} v)=0,
\end{equation}
and initial data $v(0)=u(0)+iB^{-1}(u_{t}(0))=\Phi_{N}+\Phi^{N}$ where we set $\Phi_{N}\coloneqq\phi_{N}+iB^{-1}(\psi_{N})\in H^{1}(\R^d)$ and $\Phi^{N}\coloneqq\phi^{N}+iB^{-1}(\psi^{N})\in M_{r,r'}^{1}(\R^d)$. From \eqref{bigsplit} we have the norm estimates
\begin{equation}
\label{normPhi}
\|\Phi_{N}\|_{H^{1}}\lesssim N^{\tilde{\alpha}},\ \|\Phi^{N}\|_{M^{1}_{r,r'}}\lesssim\frac1{N},
\end{equation}
and more generally, 
\begin{equation}
\|(u(t,\cdot),u_{t}(t,\cdot))\|_{H^{1}\times L^{2}+M^{1}_{2\alpha,(2\alpha)'}\times M_{2\alpha,(2\alpha)'}}\approx\|v(t,\cdot)\|_{H^{1}+M^{1}_{2\alpha,(2\alpha)'}}.
\end{equation}
The Hamiltonian of \eqref{newmaineq} (see \cite[Equation IV.2.7]{BOU}) is formally given by the formula
\begin{equation}
\label{Hamm}
H(v)=\int_{\R^d}\Big[\frac12\ |Bv|^{2}+\frac1{\alpha+1}\ |\operatorname{Re} v|^{\alpha+1}\Big]\ dx.
\end{equation}
However, we cannot use it to control the full solution $v$, since \eqref{Hamm} does not make sense for general $v(t,\cdot)\in H^{1}+M^{1}_{2\alpha,(2\alpha)'}$. From \eqref{boundd} the linear evolution $e^{-itB}\Phi^{N}$ never blows up in modulation spaces on any bounded time interval $I=[0,T]$. More precisely we have that
\begin{equation}
\label{lcrp77}
\|e^{-itB}\|_{M^{1}_{2\alpha,(2\alpha)'}(\R^d)\rightarrow M^{1}_{2\alpha,(2\alpha)'}(\R^d)}\lesssim1,
\end{equation}
for all $t\in I$.
Therefore, we are left to control
\begin{equation}
\label{change22}
\tilde{v}\coloneqq v-e^{-itB}\Phi^{N}
\end{equation}
in $H^{1}$ on $I$ (see Equation \cite[IV.2.21]{BOU}) using the Hamiltonian. 

Observe that $\tilde{v}(0)=\Phi_{N}$ and $\|\tilde{v}\|_{H^{1}}\lesssim N^{\tilde{\alpha}}$. Let us also define $I(t)=H(\tilde{v}(t))$ which at $0$ is controlled by 
\begin{equation}
\label{Hamnew}
I(0)\lesssim N^{2\tilde{\alpha}}+N^{\tilde{\alpha}(\alpha+1)}\lesssim N^{\tilde{\alpha}(\alpha+1)}
\end{equation}
where we have used the Sobolev embedding \eqref{embeddsob} for the second summand. Our goal is to estimate the time $T$ (as a function of $N$) that preserves \eqref{Hamnew}, i.e. $T$ such that for all $0\leq t\leq T$ we have
\begin{equation}
\label{Hamfort}
I(t)=\int_{\R^d}\Big[\frac12\ |B\tilde{v}(t,x)|^{2}+\frac1{\alpha+1}\ |\operatorname{Re} \tilde{v}(t,x)|^{\alpha+1}\Big]\ dx\leq2\ I(0)\lesssim N^{\tilde{\alpha}(\alpha+1)}.
\end{equation}
At least formally we have
\begin{equation}
\label{derof}
I'(t)=\operatorname{Im}\Big\langle B\tilde{v}, |\operatorname{Re} v|^{\alpha-1}\operatorname{Re} v-|\operatorname{Re}\tilde{v}|^{\alpha-1}\operatorname{Re}\tilde{v}\Big\rangle
\end{equation}
(see \cite[Equation IV.2.26]{BOU}). Invoking the Cauchy-Schwarz inequality we estimate the last quantity in absolute value by
\begin{equation}
\label{CS1}
\|B\tilde{v}\|_{2}\ \||\operatorname{Re} v|^{\alpha-1}\operatorname{Re} v-|\operatorname{Re}\tilde{v}|^{\alpha-1}\operatorname{Re}\tilde{v}\|_{2}.
\end{equation}
The first factor is estimated by $I(t)$, namely
\begin{equation}
\label{Godh}
\|B\tilde{v}\|_{2}\lesssim I(t)^{\frac12}\lesssim N^{\frac{\tilde{\alpha}(\alpha+1)}2}.
\end{equation}
For the second factor we have the pointwise size estimate (see \eqref{sizz})
\begin{eqnarray}
\label{secondfactGodh}
\left||\operatorname{Re} v|^{\alpha-1}\operatorname{Re}v-|\operatorname{Re}\tilde{v}|^{\alpha-1}\operatorname{Re}\tilde{v}\right| &\lesssim& |e^{-itB}\Phi^{N}|(|\operatorname{Re} v|^{\alpha-1}+|\operatorname{Re}\tilde{v}|^{\alpha-1}) \\ &\lesssim& |e^{-itB}\Phi^{N}||\tilde{v}|^{\alpha-1}+|e^{-itB}\Phi^{N}|^{\alpha}. \nonumber
\end{eqnarray}
Therefore, by H\"older's inequality for the first summand and with the use of the embedding of modulation spaces into Lebesgue spaces \eqref{embedd} and the Sobolev embedding \eqref{embeddsob} we obtain the estimate
\begin{eqnarray}
\label{Holde}
\left\||\operatorname{Re} v|^{\alpha-1}\operatorname{Re} v-|\operatorname{Re}\tilde{v}|^{\alpha-1}\operatorname{Re}\tilde{v}\right\|_{2} \nonumber &\lesssim& \|e^{-itB}\Phi^{N}\|_{2\alpha} \|\tilde{v}\|_{2\alpha}^{\alpha-1}+\|e^{-itB}\Phi^{N}\|_{2\alpha}^{\alpha} \\ &\lesssim& \|e^{-itB}\Phi^{N}\|_{M_{2\alpha,(2\alpha)'}}\|\tilde{v}\|_{2\alpha}^{\alpha-1}+\|e^{-itB}\Phi^{N}\|_{M_{2\alpha,(2\alpha)'}}^{\alpha} \\ &\lesssim&  \|\Phi^{N}\|_{M^{1}_{2\alpha,(2\alpha)'}}\|\tilde{v}\|_{H^{1}}^{\alpha-1}+\|\Phi^{N}\|_{M^{1}_{2\alpha,(2\alpha)'}}^{\alpha} \nonumber \\ &\lesssim& \|\Phi^{N}\|_{M^{1}_{2\alpha,(2\alpha)'}}I(t)^{\frac{\alpha-1}2}+\|\Phi^{N}\|_{M^{1}_{2\alpha,(2\alpha)'}}^{\alpha} \nonumber \\ &\lesssim& N^{\frac{\tilde{\alpha}(\alpha+1)(\alpha-1)}2-1}+N^{-\alpha}. \nonumber
\end{eqnarray}
At the third inequality we used the boundedness of the operator $e^{-itB}$ on modulation spaces from Equation \eqref{boundd}.

Thus, from the mean value theorem, it follows that for $0\leq t\leq T$ and some $\tau\in[0,t]$ we have the estimate
\begin{eqnarray}
\label{meanvalth}
|I(t)-I(0)| &\leq& T|I'(\tau)|\\ &\lesssim& T N^{\frac{\tilde{\alpha}(\alpha+1)}2}\Big(N^{\frac{\tilde{\alpha}(\alpha+1)(\alpha-1)}2-1}+N^{-\alpha}\Big) \nonumber \\  &=& T\left(\frac{1}{N^{1-\frac{\tilde{\alpha}(\alpha+1)\alpha}{2}}}+\frac{1}{N^{\alpha-\frac{\tilde{\alpha}(\alpha+1)}{2}}}\right).  \nonumber
\end{eqnarray}
For \eqref{Hamfort} to be true it suffices that the last expression of \eqref{meanvalth} satisfies 
\begin{equation}
\label{finaltruebound}
T\left(\frac{1}{N^{1+\frac{\tilde{\alpha}(\alpha+1)(2-\alpha)}{2}}}+\frac1{N^{\alpha+\frac{\tilde{\alpha}(\alpha+1)}{2}}}\right)\lesssim1.
\end{equation}
Since $N$ is going to be large we want that both exponents of $N$ in the last expression are positive. Therefore, we require
\begin{equation}
\label{alphareq}
\frac{\tilde{\alpha}(\alpha+1)(\alpha-2)}{2}<1.
\end{equation}
The last condition is satisfied due to the assumption $p\in\left(2,p_{\text{max}}\right)$. Also, it is straightforward to see that 
\begin{equation}
\label{estimaat4}
1+\frac{\tilde{\alpha}(\alpha+1)(2-\alpha)}{2}\leq\alpha+\frac{\tilde{\alpha}(\alpha+1)}{2}.
\end{equation}
Hence, \eqref{Hamfort} holds for $T\sim N^{1+\frac{\tilde{\alpha}(\alpha+1)(2-\alpha)}{2}}$ and by putting everything together we obtain for $0\leq t\leq T$
\begin{equation}
\label{finalestii}
\|v(t)-e^{-itB}v(0)\|_{H^{1}(\R^d)}<I(t)^{\frac12}+\|\Phi_{N}\|_{H^{1}(\R^d)}\lesssim N^{\frac{\tilde{\alpha}(\alpha+1)}{2}}\sim T^{\frac{\tilde{\alpha}(\alpha+1)}{2+\tilde{\alpha}(\alpha+1)(2-\alpha)}}
\end{equation}
or in other words
\begin{equation}
\label{finn74}
\|(u(t),u_{t}(t))-\mathbb K(t)(u(0),u_{t}(0))\|_{H^{1}\times L^{2}}\lesssim (1+t)^{\frac{\tilde{\alpha}(\alpha+1)}{2+\tilde{\alpha}(\alpha+1)(2-\alpha)}}
\end{equation}
for all $t\in[0,T]$ which finishes the proof of Theorem \ref{maintheoremglo}.

\end{proof}

\textbf{Acknowledgments}:
Funded by the Deutsche Forschungsgemeinschaft (DFG, German Research
Foundation) – Project-ID 258734477 – SFB 1173.

The authors would like to thank Dirk Hundertmark and Peer Kunstmann from
KIT for their helpful comments and fruitful discussions.

\end{section}

{\let\thefootnote\relax\footnote{Current address of Leonid Chaichenets: Technical University of Dresden, Institute of Analysis, 01069 Dresden, Germany, email:\textbf{leonid.chaichenets@tu-dresden.de}}}

\end{document}